\documentclass[10pt]{article}
\usepackage{amssymb,amsfonts}
\usepackage{pb-diagram}
\usepackage{amsmath,amsthm,amsfonts,amssymb}
\usepackage[usenames]{color}
\usepackage{hyperref}
\usepackage{soul}
\usepackage{graphicx}
\usepackage{amssymb}
\usepackage{amsmath}
\usepackage{amssymb}

\newtheorem{theorem}{Theorem}
\newtheorem{definition}{Definition}

\newtheorem{lemma}{Lemma}

\newcommand{\be}{\begin{enumerate}}
\newcommand{\ee}{\end{enumerate}}
\newcommand{\beq}{\begin{equation}}
\newcommand{\eeq}{\end{equation}}

\def\N{{\mathbb{N}}}
\def\Z{{\mathbb{Z}}}

\def\A{{\mathcal{A}}}
\def\M{{\mathcal{M}}}

{}
{}
\title{Undecidability of Equations in Free Lie  Algebras}
\author{Olga Kharlampovich \footnote{Hunter College, CUNY, Supported by  PSC-CUNY  award} and Alexei Myasnikov \footnote{Stevens Institute of Technology, the results of this paper were obtained with  support of RSF grant (project No. 14-11-00085)}}

\date{}

\begin{document}

\maketitle

\begin{abstract} In this paper we prove undecidability of finite systems of equations   in free Lie algebras of rank at least three over an arbitrary  field.  We show that the ring of integers $\mathbb{Z}$ is interpretable by positive existential formulas   in such free Lie algebras over a field of characteristic zero. \end{abstract}

\tableofcontents

\section{Introduction} 

In this paper we prove undecidability of finite systems of equations   in free Lie algebras $L$ of rank at least three over an arbitrary   field $K$. Furthermore, we show that the ring of integers $\mathbb{Z}$ is  interpretable by positive existential formulas  in such free Lie algebras over a field characteristic zero. We also prove that the field $K$ is interpretable by finite systems of equations in free Lie algebras $L$ of rank at least 2.  

Study of algebraic equations is one of the oldest and most celebrated themes in mathematics. It was understood that finite systems of equations with coefficients in $\mathbb C$ or $\mathbb R$ are decidable. The celebrated Hilbert tenth problem stated in 1900 asks for  a procedure which, in a finite number of steps, can determine whether a polynomial equation (in several variables) with integer coefficients has or does not have integer solutions.
In 1970  Matiyasevich, following the work of Davis, Putnam and Robinson, solved this problem in the  negative  \cite{mat}.
Similar  questions can be asked for arbitrary commutative rings $R$.  The {\em Diophantine problem for a given commutative ring $R$} asks if there exists  an algorithm that decides whether or not a given polynomial equation (a finite system of polynomial equations) with coefficients in some subring $R_0$ of  $R$  has a solution in $R$. In this case elements of $R_0$ must be recognizable by computers, so we always assume that $R_0$ is a computable ring.
One can consider equations and their decidability over arbitrary algebraic structures  $\mathcal{M}$  in a language $L$. 
An equation in  $\mathcal{M}$ is an equality of two terms in $L$:
$$
t(x_1, \ldots, x_n, a_1, \ldots, a_m) = s(x_1, \ldots, x_n, b_1, \ldots, b_k).
$$
with variables $x_1, \ldots,x_n$ and constants $a_1, \ldots, a_m, b_1, \ldots, b_k  \in \mathcal{M}$.  A solution of such an equation is a map  $x_i \to c_i$ from the set of variables into $\mathcal{M}$ which turns the symbolic equality of terms $t = s$ into a true   equality in $\mathcal{M}$. In particular,  one can consider equations in semigroups, groups,  associative or Lie algebras, etc. 
 
The  undecidability of the Diophantine problem was proved for rings of polynomials in one variable $R[X]$ over an integral domain $R$ \cite{denef,denef2},  for  rings of Laurent polynomials $R[X,X^{-1}]$  \cite{pappas,pheidas},  for free associative algebras over any field   and for   group algebras over any field and of  a wide variety  of  torsion free groups, including toral relatively hyperbolic groups, right angled Artin groups, commutative transitive groups, the  fundamental groups of various graph of groups,  etc. \cite{KM1}.

A major open problem  is the Diophantine problem (sometimes called  generalized  Hilbert's tenth problem) for the field ${\mathbb Q}$ of rational numbers (see  a comment on this  in the next section). A  survey of the results on the undecidability of existential theories of rings and fields is given in \cite{pheidas_zahidi}.

The following are the principal    questions on equations in  $\mathcal{M}$:
decidability of single equations and finite systems (Diophantine problems),
equivalence of  infinite systems of equations  in finitely many variables  to some of their  finite subsystems (equationally Noetherian structures),
description of solution sets of finite systems of equations.  The principal questions are solved positively in 
abelian groups (linear algebra), 
 free groups \cite{Makanin},\cite{Razborov}, \cite{irc},
hyperbolic  and toral relatively hyperbolic groups \cite{Rips-Sela},\cite {Dahmani-Groves},
Right angled Artin groups \cite{diek}, \cite{CRK} and free products of groups  with decidable  equations in the factors  (see also  \cite{CRK1}), and some other groups. Not much is known about description of solutions sets of finite systems of equations in free associative algebras, and whether such algebras are equationally Noetherian or not.

  As we mentioned above in this paper we prove the undecidability of equations   for free Lie algebras $L$ with basis $A = \{a,b,c, \ldots\}$  over an arbitrary  field $K$.  To do so we first study  solutions in $L$ of the linear system of equations $ [x,c]+[y,b]=[z,a], [x,b]=[z_1,a],  [y,c]=[z_2,a]$ in variables $x, y, z, z_1, z_2$ and constants $a,b,c \in L$. Then we describe precisely the Diophantine  set which is the projection of the solution set of the system onto the first two  coordinates $(x,y)$ and using this description we interpret by positive existential formulas the ring of polynomials $K[t]$ in one variable $t$. This, together with  Denef's  results \cite{denef,denef2} on undecidability of the Diophantine problems in $K[t]$, gives undecidability of the Diophantine problem in $L$. In fact, one can write down an infinite sequence of finite systems of equations   $S(X,A,u) = 0$ in variables $X$ and coefficients in $L$ where the coefficient $u$ runs over a particular  infinite set of elements in $L$ (the parameter of the system)  for which there is no algorithm to decide if a given system $S(X,A,u) = 0$ has a solution in $L$ or not. 
  
 Note that describing solution sets of equations in $L$ is a very difficult task.  Solution sets of linear equations of the type $[x,u]+ [y,v] = 0$ in variables $x,y$ and constants $u,v \in L$ were described in \cite{RS}, later these results were generalised to arbitrary linear equations of the type $[x_1,u_1]+ \ldots + [x_n,u_n] = 0$  (with some restrictions on the coefficients $u_1, \ldots,u_n \in L$ and sometimes requiring characteristic zero of the field $K$).  The general theory of equations in free Lie algebras is not developed yet, though the fundamentals  of the algebraic geometry over arbitrary Lie algebras were outlined in the works of Daniyarova,  Kazachkov, and Remeslennikov \cite{DKR1,DKR2,Kaz,DR}. In particular, in  the paper  \cite{DR} systems of equations over free Lie algebras whose solution sets are contained in a finite-dimensional affine subspace were studied. 
 
 We finish this introduction by stating  some open questions:
  
 \begin{enumerate}
\item [1)] Is a finite system of equations in $L$ equivalent to a single equation?   
\item [2)] Is $L$ equationally Noetherian?  
\item [3)]  Prove the undecidability of equations  for free Lie algebras of rank 2 over fields.
\end{enumerate}

\section{E-interpretability}

Recall that $A\subset \mathcal{M}^n$ is called {\em e-definable} (equationally definable) or {\em Diophantine}  in $\mathcal{M}$ if there exists a finite system of equations $\Sigma(x_1,\ldots,x_n,y_1, \ldots,y_m)$ such that 
$(a_1,\ldots,a_n) \in A$ if and only the system  $ \Sigma(a_1,\ldots,a_n,y_1, \ldots,y_m)$ in variables $y_1, \ldots,y_m$ has a solution in $\mathcal{M}$. In other words,  Diophantine sets are projections of algebraic sets defined by finite systems of equations.

Now we define an important notion of interpretations by equations.

\begin{definition}(E-interpretation or Diophantine interpretation)
Let $\mathcal{A}$ and $\mathcal{M}$ be algebraic structures. A map $\phi:X\subset \mathcal{M}^n \to \mathcal{A}$ is called an e-interpretation of $\mathcal{A}$ in  $\mathcal{M}$ if 
\begin{itemize}
\item [1)] $\phi$ is onto;
\item [2)] $X$ is e-definable in $\mathcal{M}$;
\item [3)] The preimage of $"="$ in $\mathcal{A}$ is e-definable in $\mathcal{M}$; 
\item [4)] The preimage of the graph of every function and predicate  in $\mathcal{A}$ is e-definable in $\mathcal{M}$.
\end{itemize}
\end{definition}
For algebraic structures $\mathcal A$ and $\mathcal M$ we write $\A \to_e \M$ if $\A$ is e-interpretable in $\M$. 

\medskip
\noindent
{\bf Examples.} {\it The following are known examples of e-interpretability: 
\begin{itemize}
\item [1)] $\N$ is Diophantine in $\Z$ since 
$$  x \in \mathbb{N} \Longleftrightarrow \exists y_1, \ldots, y_4 (x = y_1^2 + \ldots +y_4^2);$$ 
\item [2)] the ring $\mathbb Q$ is e-interpretable in $\mathbb Z$ as a field of fractions. 
\item [3)] $\mathbb Z$ is e-interpretable in $\mathbb N$;
\item [4)] The structure $\langle \mathbb Z ; +, \mid \rangle$, where $\mid$ is the predicate of division, is e-iterpretable in $\mathbb Z$.
\end{itemize}
}

The following result is easy, but useful.
\begin{lemma}(Transitivity of Diophantine interpretation) \label{le:transitive}
For algebraic structures  $A, B, C$  if $A \to_eB$ and $ B \to_e C$ then $A \to_e C$.
\end{lemma}

The following result gives the main technical tool in our study of Diophantine problems.
\begin{lemma}\label{le:reduction}
  Let $\phi:X\subset \mathcal{M}^n\to \mathcal{A}$ be a Diophantine interpretation of $\mathcal{A}$ in $\mathcal{M}$. Then  there is an effective procedure that given   a finite system of equations $S = 1$ over $\mathcal{A}$ constructs an equivalent system of equations $S' = 1$ over $\mathcal{M}$, such that  $\bar a$ is a solution of $S = 1$ in $\mathcal{A}$ iff  $\phi^{-1}(\bar a)$ is a solution of $S' = 1$ in $\mathcal{M}$.
  \end{lemma}

\begin{lemma}\label{le:reductions}
 Let $\A$ and $\M$ be algebraic structures such that $\A \to_e \M$. Then the following holds:
\begin{itemize}
\item [1)] if the Diophantine problem  in $\mathcal{A}$ is undecidable then it is undecidable in $\mathcal{M}$.
\item [2)] if the Diophantine problem  in $\mathcal{M}$ is decidable then it is decidable in $\mathcal{A}$.
\end{itemize}
\end{lemma}

\section{Maximal rings of scalars and algebras}

Let  $A = \{a_1, \ldots,a_n\}$, $n \geq 2$,  be a finite set  and $K$ a field. $L$ be a free  Lie $K$-algebra with basis $A$. It is known that the maximal ring of scalars is the field $K$  and the field $K$ is 0-interpretable in $L$ \cite{Lie}. In this section we show that the field $K$ is interpretable in $L$ by finite systems of equations with coefficients in $A$.

 We denote a tuple $(y_1, \ldots,y_n)$ by $\bar y$, and the same for $\bar x, \bar z, \bar \alpha$, etc.
By $[x,y]$  we denote  the multiplication in $L$.

1). Consider a set 
$$
A = \{(y_1, \ldots,y_n) \in L^n \mid \bigwedge_{i = 1}^n [y_i,a_i] = 0\}.
$$ 
Since for any $y \in L$, $(y,a_i) = 0$ if and only if $y = \alpha a_i$ for some $\alpha \in K$, it follows that 
$$
A = \{(\alpha_1 a_1, \ldots, \alpha_n a_n) \mid \alpha_i \in K\}.
$$
Notice that $A$ is defined by a finite system of equations in $L$.

2). Put 
$$
A_0 = \{ \bar y \in A \mid \bigwedge_{i,j=1}^n [y_i,a_j] = [a_i,y_j]\}.
$$
The set $A_0$ is defined by a finite system of equations in $L$. Notice that 
$$
A_0 = \{(\alpha a_1, \ldots, \alpha a_n) \mid \alpha \in K\}.
$$
We denote the tuple $(\alpha a_1, \ldots, \alpha a_n)$ by $\alpha \bar a$. From now on we    interpret an element $\alpha \in K$ by the  tuple $\alpha \bar a$, so $K$ is interpreted as the set $A_0$.

 Now we interpret operations of addition $+$ and multiplication $\cdot$ from $K$ on the set $A_0$. 
 
 3) Let $\bar x, \bar y, \bar z \in A_0$. Then $\bar x = \alpha \bar a, \bar y = \beta \bar a, \bar z  = \gamma \bar a$ for some $\alpha, \beta, \gamma \in K$. Notice that $\alpha +\beta  = \gamma$ in $K$ if and only if $\bar x, \bar y, \bar z$ satisfy in $L$ the following system of equations:
 $$
 \bigwedge_{i = 1}^n (x_i +y_i = z_i).
 $$
 This interprets addition $+$ on $A_0$.  Note that it suffices to write only one equation $x_1 +y_1 = z_1$.
 
 4) Observe now that in the notation above $\alpha \cdot \beta  = \gamma$ if and only if $\bar x, \bar y, \bar z$ satisfy in $L$ the following system of equations:
 
 $$
 \bigwedge_{i,j = 1}^n [x_i,y_j]  = [z_i,a_j].
 $$
 In fact, it suffices to write only one equation $ [x_1,y_2]  = [z_1,a_2]$. This interprets multiplication $\cdot$ on $A_0$ by equations.
 
5)   Now we interpret the action of $K$, viewed as the set $A_0$,  on $L$ by a finite system of equations in $L$. 

Let $x,z \in L$ and $\bar y = \alpha \bar a \in A_0$. Then $z = \alpha x$ in $L$ if and only if $x,z$ and the components $y_1, \ldots, y_n$ of $\bar y$ satisfy the following system of equations:
$$
\bigwedge_{i = 1}^n [z,a_i] = [x,y_i].
$$
 Indeed,  $[z,a_i] = [x,y_i]$ implies that $[z,a_i] = [x,\alpha a_i]$, so $[z-\alpha x,a_i] = 0$ for $i = 1, \ldots, n$. Since $Ann(a_1, \ldots, a_n) = 0$  it follows that $z- \alpha x = 0$, so $z = \alpha x$ as required. 
  
  We summarize the argument above in the following theorem
  
  \begin{theorem} \label{th:2}
  Let $A = \{a_1, \ldots,a_n\}$, $n \geq 2$,  be a finite set,  $K$ a field, and  $L$ a free  Lie $K$-algebra with basis $A$. Then the field $K$ and its action on $L$ are defined by finite systems of equations in $L$.
  \end{theorem}

Now Theorem \ref{th:2} and Lemma \ref{le:reductions} imply the following result.

\begin{theorem}
Let $A = \{a_1, \ldots,a_n\}$, $n \geq 2$,  be a finite set,  $K$ a field, and  $L$ a free  Lie $K$-algebra with basis $A$.  If the Diophantine problem in the field $K$ is undecidable then it is undecidable in the algebra $L$.

\end{theorem}

\section{E-interpretability of the arithmetic in a  free Lie algebra of rank $> 2$}

Let $A=\{a,b,c,  a_1,\ldots, a_n \}$ be a finite set with $|A| \geq 3$, $K$ an integral domain, and $L$ a free Lie algebra with  basis  $A$ and  coefficients in  $K$.

For a subset $ S \subseteq L$ by $\langle S\rangle$ we denote the $K$-subalgebra of $L$ generated by $S$. By  $[z_1,z_2,\ldots ,z_n]$  we denote the left-normed product of elements $z_1,z_2,\ldots ,z_n$ in $L$. 
For $u,v \in L$ and $\alpha \in R$ by $[u,(v+\alpha)]$ we denote the element $[u,v] +\alpha u \in L$ and refer to it as a "product" of $u$ and $v+\alpha$. We denote the product $[u,v,\ldots ,v]$, where $v$ occurs $n$ times, by $[u,v^{(n)}]$, in particular, $[u,v^{(0)}] = u$. Now if $f(t)=\alpha _nt^{n}+\ldots +\alpha _0$ is a polynomial in one variable $t$ and coefficients in $K$, then by $[u,f(a)]$ we denote the element $\alpha _n[u,a^{(n)}]+\ldots +\alpha _0u$ from $L$. The endomorphism $ad(a): x \to [x,a]$ of the $K$-vector space $L^+$ generates in $End(L^+)$ a $K$-subalgebra $\langle ad(a)\rangle_K$ which isomorphic to the ring of polynomials $K[a]$, so   $K[a]$  acts on $L$ as a subalgebra of $End(L^+)$. Now, the element $[u,f(a)]$ is precisely the result of the action of the polynomial $f(a) \in K[a]$ on the element $u$.

Recall that 
a Hall basis $\mathcal H$ of $L$ with respect to the set of generators  $A$  is a  linearly ordered  set of homogeneous elements of $L$ such that $\mathcal H = \cup_{n=1}^\infty H_n$, where $H_n$ are   sets of all elements in $\mathcal H$ of degree $n$ which  are  defined by induction as follows.

I) The set $H_1$ of all elements
of degree 1 in $\mathcal H$ is  precisely the set of free generators from $A$  with some a priori given order. 

II) Suppose the set $H_{n-1}$ of all  elements of degree $n-1$ in  $\mathcal H$ is defined and the set $\cup_{i = 1}^{n-1}H_i$ is linearly ordered. Then the set $H_n$ of all elements of degree $n>1$  in $\mathcal H$ consists  precisely of the elements
$[e, f]$,  where $e, f \in \cup_{i = 1}^{n-1}H_i$  such that 

(i) $deg\  e + deg\ f = n$, 

(ii)
$e > f$, 

(iii) $deg\ e = 1$ or, if $e = [e_1, e_2]$ with $e_1, e_2 \in \cup_{i = 1}^{n-1}H_i$, then $e_2\leq f$. 

The order on $\cup_{i = 1}^{n}H_i$ extends the order on $\cup_{i = 1}^{n-1}H_i$ in such a way that the order on $H_n$ is an arbitrary linear order and for all $e, f \in \cup_{i = 1}^{n}H_i$, $deg\ e > deg\  f$ implies $e > f$.   By induction all the sets $H_n$ are defined and we put $\mathcal H = \cup_{n=1}^\infty H_n$.

 We order generators from $A$ in such a way that $a < b < c$. Then elements $[b,a^{(k)}]$  belong to the Hall basis $\mathcal H$  of $L$ with respect to $A$. If $k>p$ then the elements  $[[b,a^{(k)}],[b,a^{(p)}]]$ also belong to $\mathcal H$.  From now on we fix a particular Hall basis $\mathcal H = {\mathcal H}_A$ of $L$ with respect to the basis $A$. It is known (see for example, \cite{Bahturin,Magnus}) that $\mathcal H$ is a $K$-linear basis of $L$.

 \begin{lemma} \label{le:faithful} Let $f \in K[t]$. Then  $[b,f(a)]=0$ implies that $f = 0$ in $K[t]$.   
 \end{lemma}
 \begin{proof}  Let $f =\alpha _nt^{n}+\ldots +\alpha _0 \in K[t]$. Then by definition $[b,f(a)]  = \alpha_n[b,a^{(n)}] + \ldots +\alpha_0$ is a linear combination of elements $[b,a^{(k)}]$  that belong to the Hall basis of $L$. Hence if $[b,f(a)] = 0$ then all the coefficients in the linear combination are equal to zero, so $f = 0$, as claimed.
 \end{proof}

\begin{lemma}\label{le:ua} Let  $u$ be an element of $L$ such that $u \in \langle a,b\rangle$, $u$ is homogeneous in $b$ with $deg_b u = 2$, and $u$ is homogeneous in $a$ with $deg_a u = 2n$. Then  $[u,a]\neq [b,a^{(2n+1)},b]. $
\end{lemma}

\begin{proof}  Let $u$ be as above. Then $u \in \langle a,b\rangle$ and $\langle a,b\rangle$ is a free Lie $K$-algebra with basis $a,b$, so $u = \sum_i \alpha_i u_i$ where $\alpha_i \in K$ and $u_i \in {\mathcal H}_A$. Since $u$ is of degree 2 in $b$ it follows that each $u_i$ is a product of two elements of the type $[b,a^{(k)}]$. Furthermore, since $u$ is of degree $2n$ in $a$ it follows that  $u_i$ are elements of the type $[[b,a^{k}],[b,a^{(2n-k)}]]$, where $n+1\leq k\leq 2n$. Hence 
$$u=\sum _{k=n+1}^{2n} \alpha _k[[b,a^{k}],[b,a^{(2n-k)}]].$$
   Notice that by Jacoby identity
$$
[[b,a^{(k)}],[b,a^{(2n-k)}],a] = [[b,a^{(k+1)}],[b,a^{(2n-k)}]] + [[b,a^{(k)}],[b,a^{(2n-k+1)}]].
$$
Then 
$$
[u,a]=\alpha _{n+1}[[b,a^{(n+1)}],[b,a^{(n)}]]+\sum _{k=n+1}^{2n-1} (\alpha _k+\alpha _{k+1})[[b,a^{k+1}],[b,a^{(2n-k)}]]+\alpha _{2n}[b,a^{(2n+1)},b]
$$ 
and this is a linear combination of basis elements from ${\mathcal H}_A$. Suppose now, that $[u,a]=[[b,a^{(2n+1)}],b]$.
It follows that the linear combination above  equals to the basis element $[[b,a^{(2n+1)}],b]$, which implies that  $\alpha _{n+1}=0$ and, therefore, $\alpha _k=0$ for $n+1\leq k\leq 2n$, i.e., $u=0$, contradicting the assumption that $u$ is of degree 2 in $b$. This proves the lemma.

\end{proof}

\begin{lemma} \label{le:principal}
For any $r\in L$ and any $m,n \in \mathbb{N}$  there exists $s \in L$ such that 
\begin{equation}\label{jacoby} 
[[r,a^{(m)}],[b,a^{(2n)}]]= [[r,a^{(m+2n)}],b]+ [s,a].
\end{equation}  
and there exists $t \in L$ such that 
\begin{equation}\label{jacoby-} 
[[r,a^{(m)}],[b,a^{(2n+1)}]]= -[[r,a^{(m+2n+1)}],b]+ [t,a].
\end{equation}  

\end{lemma}

\begin{proof}
Note that the Jacoby identity implies the  identity of the form 
\begin{equation} \label{eq:almostJacoby}
[u,[v,w]]=[[u,v],w]-[[u,w],v].
\end{equation}

Using the identity (\ref{eq:almostJacoby}) we will prove the equality  (\ref{jacoby})  by induction on $n$, i.e., we show  that for any $r\in L$ and any $n,m \in \mathbb{N}$ 
$$
[[r,a^{(m)}],[b,a^{(2n)}]]= [[r,a^{(m+2n)}],b]+ [s,a],
$$ 
for some $s\in L.$ 

It is obvious for $n=0$ and arbitrary $m\geq 0$, since in this case the equality (\ref{jacoby}) holds for $ s= 0$. 

Assume now that   (\ref{jacoby}) holds  for  any number $\ell$ less than $n$ and any $m$.

We put $u=[r,a^{(m)}], v=[b,a^{2n-1}], w=a$ and use the Jacobi identity in the  form (\ref{eq:almostJacoby}). Then
$$
[[r,a^{(m)}],[b,a^{(2n)}]]= [[r,a^{(m)}],[b,a^{(2n-1)}],a]-[[r,a^{(m+1)}],[b,a^{(2n-1)}]].
$$ 

Applying the identity (\ref{eq:almostJacoby}) again to the last term above 
we obtain 
$$
[[r,a^{(m)}],[b,a^{(2n)}]]= [[r,a^{(m)}],[b,a^{(2n-1)}],a]-[[r,a^{(m+1)}],[b,a^{(2n-2)}],a]+[[r,a^{m+2}],[b,a^{(2n-2)}]].
$$
By induction 
$$
[[r,a^{(m+2)}],[ba^{(2n-2)}]]=[[r,a^{(m+2n)}],b]+[s_1,a]
$$ 
hence 
$$
[[r,a^{(m)}],[ba^{(2n)}]] = [[r,a^{(m+2n)}],b] + [s,a],
$$
where 
$$
s=[[r,a^{(m)}],[b,a^{(2n-1)}]]-[[r,a^{(m+1)}],[ba^{(2n-2)}]]+s_1.
$$ 
This proves  (\ref{jacoby}). 

To show  (\ref{jacoby-}) rewrite the element $[[r,a^{(m)}],[b,a^{(2n+1)}]]$  using (\ref{eq:almostJacoby}) with $u=[r,a^{(m)}], v=[b,a^{2n}], w=a$. This results modulo $[L,a]$ in the equality
$$
[[r,a^{(m)}],[b,a^{(2n+1)}]]= -[[r,a^{(m+1)}],[b,a^{(2n)}]] 
$$
Rewriting the right-hand side using (\ref{jacoby}) one gets (\ref{jacoby-}) modulo $[L,a]$, as required. 
\end{proof}

\begin{lemma} \label{l1} Let $L$ be a free Lie algebra with basis $A = \{a,b,c, \ldots \}$ and coefficients in a field $K$.  Denote by $\phi(x,y,z,z_1,z_2)$ the following system of equations in variables $x,y,z,z_1,z_2$ and constants $a,b,c$ (which can be viewed as a conjunction of equations):
 $$
 [x,c]+[y,b]=[z,a] \wedge [x,b]=[z_1,a] \wedge [y,c]=[z_2,a].
 $$
  The positive existential formula
\begin{equation} \label{eq:1} 
\Psi (x,y)= \exists  z, z_1z_2 \phi(x,y,z,z_1,z_2)
\end{equation} 
defines in $L$ a  set 
$$
S =  \{ ([b,f(a^2)]+\alpha a, [c,f(a^2)]+\beta a) \mid f \in K[t], \alpha, \beta \in K\},
$$   
 here $f(a^2)$ is a polynomial in $K[a]$ all of whose odd powers have zero coefficients.
 
  \end{lemma}

\begin{proof}  
Observe that the truth set $\phi(L)$ of the formula $\phi(x,y,z,z_1,z_2)$ in $L$ is a $K$-linear subspace of $L^5$. Similarly, the truth set $\Psi(L)$ of the formula $\Psi$, as a natural projection of $\phi(L)$, is  a $K$-linear subspace of $L^2$. 
Hence to show that  $S \subseteq \Phi(L)$ it suffices to prove   that the set $\Psi(L)$ contains all the pairs $([b,a^{(2n)}], [c,a^{(2n)}])$ and $(\alpha a, \beta a)$.   The latter is easy, since $(\alpha a, \beta a, -\alpha a -\beta a, -\alpha a, -\beta a)$ is a solution to the system $\phi$  for any $\alpha, \beta \in K$.  

Now we show that $([b,a^{(2n)}], [c,a^{(2n)}]) \in \Psi(L)$ for any $n \in \mathbb{N}$. 
To do this we use the identity (\ref{eq:almostJacoby}) from Lemma \ref{le:principal}, which tells one that   
$$
[[r,a^{(m)}],[b,a^{(2n)}]]= [[r,a^{(m+2n)}],b]+ [s,a]
$$  
 for some $s \in L$.

The identity $[u,u]=0$ implies $$0=[[b,a^{(n)}],[b,a^{(n)}]], \ n\in\mathbb N.$$

We can rewrite  the  term on the right using  equality (\ref{jacoby}) or (\ref{jacoby-})  where $m = n$  and obtain
$$
[[b,a^{(2n)}],b]+[u,a]=0
$$ 
for some $u\in L.$ Therefore elements $x=[b,a^{(2n)}]$ are solutions to the equation
$[x,b]=[z_1,a]$.  Similarly, the elements $y=[c,a^{(2n)}]$ are solutions to the equation
$[y,c]=[z_2,a].$  Now to prove  that a  pair $(x,y)$, where  $x=[b,a^{(2n)}]$, $y=[c,a^{(2n)}]$ satisfies  the formula $\Psi(x,y)$ it suffices to show  that this pair satisfies the equation $[x,c]+[y,b]=[z,a]$.  

Consider  the anti-commutativity identity  
$$0=[[b,a^{(2n)}],c]+[c,[b,a^{(2n)}]],$$
 where $k,n\in\mathbb N$. Using (\ref{jacoby}) with $m = 0$ one can rewrite the term on the right in the identity above in the form
$$[[b,a^{(2n)}],c]+[[c,a^{(2n)}],b]+[v,a]=0$$
for some $v \in L$.  Therefore any pair $(x,y)$ where $x=[b,a^{(2n)}]$, $y=[c,a^{(2n)}]$ satisfies the equation $[x,c]+[y,b]=[z,a]$, hence  the formula $\Psi(x,y)$.

We will now show that each homogeneous element in  $w \in L$ satisfying $\Phi (x)=\exists y\Psi(x,y)$  is equal either to  $ \alpha a$ or to $\alpha [b,a^{(2n)}]$ for some $\alpha \in K, n \in \mathbb{N}$.  Let $x = u, y = v$ be a solution to $\Psi(x,y)$ for some $u,v \in L$.
By \cite{RS}, Theorem 5.1, since $x = u$ satisfies the equation $[x,b]=[z_1,a]$ for some $z_1 \in L$ it follows that $u$ belongs to the subalgebra $\langle a,b\rangle$ of $L$ generated by $a,b$. Similarly, since  $y = v$ satisfies  the equation $[y,c]=[z_2,a]$ for some $z_2 \in L$ it follows that $v \in \langle a,c\rangle$.   Furthermore, if $x = u$ is a solution to $\Phi (x)$ then every poly-homogeneous component of $u$ with respect to each fixed degree of $a$ and $b$ is also a solution to $\Phi (x)$.  Therefore, we can assume that $u$ and $v$ are poly-homogeneous elements in $L$. 
Let us now look at the equation $[x,c]+[y,b]=[z,a]$. Let $x = u, y = v, z = w$ be a solution to this equation in $L$, where $u$ and $v$ are poly-homogeneous.
If $deg_b(u) = 0 $ then $u = \alpha a$, as required.  If each monomial in $x = u$ contains $b$ only once, then $u=[b,a^{(m)}]$. By Lemma \ref{le:ua},  $x=[b,a^{(2n+1)}]$ cannot be a solution of $\Phi (x)$, because there is no such $z$ that $[b,a^{(2n+1)},b]=[z,a].$ Hence, in this case  $u=[b,a^{(2m)}]$, as claimed. If each monomial in $x = u$ contains $b$ 
$k$ times, $k>1$, then the poly-homogeneous elements $[u,c]$ and $[v,b]$ are in different poly-homogeneous components of $L$ (since $[u,c]$ contains $b$ $k > 1$ times, while $[v,b]$ contains $b$ only once). Hence 
$w=t_1+t_2$ and 
$[u,c]=[t_1,a], [v,b]=[t_2,a]. $ It follows as above that $t_1 \in \langle a,c\rangle$  and $u$ does not contain $b$ - contradiction.  We showed that every homogeneous component of  a solution $x = u$ to the formula  $\Phi (x)$ is equal to $u=[b,a^{(2m)}]$ for some $m \in \mathbb{N}$.
Hence for every solution $x = u, y = v$ to the  formula $\Psi(x,y)$ one has $u = [b,f(a^2)]+\alpha a$ for some one-variable  polynomial $f$ and $\alpha \in K$.  A similar argument shows that $v = [c,g(a^2)] + \beta a$ for some polynomial $g$. To finish the proof we need to show that $v = [c,f(a^2)]+\beta a$ for some $\beta \in K$.

We already showed that the pair $x=[b,f(a^2)]+\alpha a, y =[c,f(a^2)]$
satisfies $\Psi$. Since the solution set of the formula $\Psi(x,y)$ is a linear subspace of $L$ one can subtract one pair of solutions from another pair, and the resulting pair is still a solution to $\Psi(x,y)$. It follows that the pair $x = 0, y = v- [c,f(a^2)]$ is also a solution to $\Psi(x,y)$, i.e., $x= 0,  y = [c,g(a^2)] +\beta a -  [c,f(a^2)]$ is a solution to $[x,c]+[y,b]=[z,a]$. Hence $[[c,g(a^2)] +\beta a -  [c,f(a^2)],b] = [z,a]$. Which implies that $[[c,g(a^2)] -  [c,f(a^2)],b] = [z^\prime,a]$ for $z^\prime  = z +\beta b$. It follows that  $[[c,(g(a^2)-f(a^2))],b] = [z^\prime,a]$.  Again, one can see that $z^\prime \in \langle a,b\rangle$, so the element $[z^\prime,a]$ does not contain $c$. However, the element  
$[[c,(g(a^2)-f(a^2))],b]$ is a linear combination of basis elements of the type $[[c,a^{(2m)}],b]$ which contain $c$. This may happen only if $[c,(g(a^2)-f(a^2))]  =0$, but then by Lemma \ref{le:faithful}  $g(a^2)-f(a^2) = 0$, as claimed. This proves the lemma.

 \end{proof}

\begin{theorem}\label{th3}
Let $L$ be a free Lie algebra of rank $> 2$  with coefficients in a field $K$. Then the polynomial ring $K[t]$  in one variable $t$ is e-interpretable in $L$.
\end{theorem}
\begin{proof}
By Lemma \ref{l1} the  positive existential formula
$$
\Psi (x,y)= \exists  z, z_1z_2 \phi(x,y,z,z_1,z_2)
$$
defines in $L$ the   set 
$$
S =  \{ ([b,f(a^2)]+\alpha a, [c,f(a^2)]+\beta a) \mid f \in K[t], \alpha, \beta \in K\}.
$$   
Hence the formula $\Phi (x) = \exists y \Psi(x,y)$ defines in $L$ the set
$$
X = \{[b,f(a^2)]+\alpha a \mid \alpha \in K, f \in K[t]\}.
$$
We define an equivalence relation $\sim$ on $X$ by the equation, namely for $u,v \in X$ 
$$
u \sim v  \Longleftrightarrow [u-v,a] = 0.
$$
For $u \in X$ by $[u]$ we denote the equivalence class of $u$ relative to $\sim$. 
 Since the centralizer of $a$ in $L$ is $Ka$ it follows that 
$$
u \sim v  \Longleftrightarrow u = v +\alpha a
$$
for some $\alpha \in K$.
 Note that $[[b,f(a^2)]+\alpha a] = [[b,f(a^2)]]$ for any $\alpha \in K$. From the above and Lemma \ref{le:ua} it follows that 
 $$
 [b,f(a^2)] +\alpha a \sim  [b,g(a^2)] +\beta a \Longleftrightarrow f(t) = g(t) \ in \ K[t],
 $$
so the equivalence class of $[b,f(a^2)]+\alpha a$ is completely defined by $f(t)$. Now we denote the equivalence class of $[b,f(a^2)]+\alpha a$ by $\bar f$. The map $f \to \bar f$ gives a bijection $K[t] \to X/ \sim$.   

Now we define an addition $\oplus$ on $X/\sim$ by an equation:
$$
[u] \oplus [v] = [w] \Longleftrightarrow w \sim u+v.
$$
Clearly, the map $^-$ preserves the addition, i.e., 
$$
\overline{f+g} = \bar f \oplus \bar g.
$$

Now we define a multiplication $\otimes$ on $X/\sim$. 

{\it Claim.} Let $f,g,h \in K[t]$. Then $fg = h$ in $K[t]$ if and only if 
\begin{equation} \label{eq:int-multipl}
[[b,f(a^2)],[c,g(a^2)]] \sim [[b,h(a^2)],c].
\end{equation}
Indeed, if $f(t)g(t)=h(t)$ then by (\ref{jacoby}) 
$$
[[b,f(a^2)],[c,g(a^2)]] =  [[b,h(a^2)],c] + [s,a],
$$
 so (\ref{eq:int-multipl}) holds. Conversely, suppose  (\ref{eq:int-multipl}) holds for some polynomials $f,g,h$. Then as above,  
  $$
  [[b,f(a^2)],[c,g(a^2)]]  \sim [[b,(fg)(a^2)],c],
  $$
so 
$$
[[b,(fg)(a^2)],c] \sim [[b,h(a^2)],c].
$$
It follows that 
$$
[[b,(fg -h)(a^2)],c] = \alpha a,
$$
 for some $\alpha \in K$.  But $[[b,(fg -h)(a^2)],c] $ is a linear combination of basis commutators of the type $[[b,a^{(2m)})],c]$, all of which contain $c$. It follows that $fg -h = 0$, as claimed.

We define multiplication $\bar f\otimes\bar g=\bar h$ if (\ref{eq:int-multipl}) holds. More precisely,  we put for $u,v,w \in X$ 

$$
[u] \otimes [v] = [w] \Longleftrightarrow \Psi(v,v^\prime) \wedge [u,v^\prime] \sim [w,c].
$$
Indeed, this is the same definition as above, since by Lemma  \ref{l1} if  $v = [b,f{(a^2)}]$ and $\Psi(v,v^\prime)$ holds  then $v^\prime  = [b,f{(a^2)}] +\alpha a$. In this case  
$$
[u,v^\prime] \sim  [[b,f(a^2)],[c,g(a^2)]],
$$
so $[u,v^\prime] \sim [w,c]$ shows that  (\ref{eq:int-multipl}) holds.

This proves that the ring of polynomials $K[t]$ in one variable $t$ is e-interpretable in $L$.

\end{proof}

\begin{theorem}\label{th:4}
Let $L$ be a free Lie algebra of rank $> 2$  with coefficients in a field $K$ of characteristic zero. Then the arithmetic $\mathbb{Z} = <\mathbb{Z} \mid +,\times,0,1>$ is e-interpretable in $L$.
\end{theorem}
\begin{proof}

By Denef's results \cite{denef},  (see more details in \cite{KM1}),  the arithmetic $\mathbb{Z} $ is e-interpretable in $K[t]$, so by transitivity of e-interpretations and Theorem \ref{th3},  $\mathbb{Z} $ is e-interpretable in  $L$, so the result follows.
\end{proof}

\begin{theorem} \label{th2}The Diophantine problem of a free Lie algebra of rank $>2$ over a field  is undecidable.
\end{theorem} 
\begin{proof}
 By Theorem \ref{th3} the ring of polynomials $K[t]$ is e-interpretable in $L$. The Diophantine problem in $K[t]$ is undecidable (see \cite{denef,denef2}), so by Lemma \ref{le:reductions} the Diophantine problem in $L$ is undecidable.

\end{proof}

\end{document}